\documentclass[twoside]{amsart}
\usepackage{amssymb}
\usepackage{mathrsfs}
\usepackage{graphicx}
\usepackage[latin1]{inputenc}
\usepackage[english]{babel}
\newcommand{\PP}{\mathbb{P}}
\newcommand{\kk}{{\bf k}}
\newcommand{\CC}{\mbox{${\mathbb C}$}}
\newcommand{\F}{\mbox{${\mathscr F}$}}
\newcommand{\sL}{\mbox{${\mathscr L}$}}
\newcommand{\sA}{\mbox{${\mathscr A}$}}
\newcommand{\sC}{\mbox{${\mathscr C}$}}
\newcommand{\G}{\mbox{${\mathscr G}$}}

\newcommand{\ZZ}{\mbox{${\mathbb Z}$}}

\newcommand{\E}{\mbox{${\mathscr E}$}}

\newcommand{\M}{\mbox{${\mathscr M}$}}

\newcommand{\tx}{T_X}

\newcommand{\cE}{\mbox{${\mathscr E}$}}
\newcommand{\cF}{\mbox{${\mathscr F}$}}

\newcommand{\End}{\text{\rm End}}

\newcommand{\ox}{\otimes}

\renewcommand{\dim}{\mathrm{dim}\,}

\def\G{\mathcal G}

\def\T{T}

\input{xypic.tex} \xyoption{all}

\newtheorem{lema}{Lemma}[section]
\newtheorem{cor}[lema]{Corollary}
\newtheorem{teo}[lema]{Theorem}

\newtheorem{prop}[lema]{Proposition}
\theoremstyle{definition}
\newtheorem{say}[lema]{}
\newtheorem{remark}[lema]{Remark}
\newtheorem{defi}[lema]{Definition}
\newtheorem{exe}[lema]{Example}

\begin{document}

\title[On degeneracy schemes of maps of vector bundles]
{On degeneracy schemes of maps of vector bundles  and applications to holomorphic foliations}

\author{Carolina Araujo}
\address{\noindent Carolina Araujo: IMPA, Estrada Dona Castorina 110, Rio de
  Janeiro, 22460-320, Brazil}
\email{caraujo@impa.br}

\author{Maur\' \i cio Corr\^ea Jr. }
\address{\noindent Maur\' \i cio Corr\^ea Jr: Universidade Federal de Vi\c cosa,
Departamento de Matem\'atica, Avenida P.H. Rolfs, s/n, 36570--000,
Vi\c cosa-MG, Brazil} \email{mauricio.correa@ufv.br}

\thanks{ }
\date{}


\begin{abstract}
In this paper we provide sufficient conditions for maps of vector bundles on smooth projective varieties
to be uniquely determined by their degeneracy schemes.
We then specialize to holomorphic distributions and foliations.
In particular, we provide sufficient conditions
for foliations of arbitrary rank on $\PP^n$ to be
uniquely determined by their singular schemes.
\end{abstract}

\maketitle

\section{Introduction}
This work was motivated by the following problem:
\begin{center}
{\it{When is a holomorphic foliation uniquely determined by its singular scheme?}}
\end{center}
This problem has its origins in the study of spaces of foliations.
When the general member of a fixed irreducible component of a space of foliations
is determined by its singular scheme, the Hilbert scheme becomes a useful tool for
describing such component. In this case, one  hopes to read off properties of foliations from geometric
aspects of their singular schemes.

A rank $k$ holomorphic foliation $\F\subset T_X$ on a complex projective manifold $X$
gives rise to
a global section $\omega\in H^0\big(X, \bigwedge^kT_X\ox (\bigwedge^k\F)^*\big)$.
The singular scheme of $\F$
is precisely the zero scheme of $\omega$.
The first observation is that
$\F$ can be recovered from the global section $\omega$ (see Remark~\ref{recovering_F}).
So we are reduced to the classical problem of reconstructing sections of
locally free sheaves from their zero schemes.
We refer to \cite{hartshorne} for a general discussion on this topic.
Often one cannot determine a
section of $\bigwedge^kT_X\ox (\bigwedge^k\F)^*$ by its zero scheme.
However, we show that, under suitable conditions, this can be done for
sections arriving from foliations.

The same approach can be used to tackle the following more general problem:
\begin{center}
{\it{When is a map of
vector bundles on a smooth projective variety determined by its degeneracy scheme?}}
\end{center}
Degeneracy schemes of maps of vector bundles  show up in a number of
geometric constructions.
Example~\ref{example:palatini} discusses an instance of this,
which has been classically studied by  Castelnuovo, Palatine and Fano, among others.

We adopt the following notation. Given a locally free sheaf $\cE$ of
rank $e$ on a variety $X$, we denote by $\cE^*$ its dual sheaf, by
$S_i\cE$  its $i$-th symmetric power, and by $\det(\cE)$ the
invertible sheaf $\bigwedge^{e}\cE$. In this paper we prove the
following criterion, which holds for varieties defined over
an arbitrary algebraically closed field.

\begin{teo}\label{main_prop}
Let $X$ be a smooth projective variety, and $\E$ and $\G$ locally free sheaves on $X$
of rank $e$ and $g$, respectively.
Let $\varphi:\E\to \G$ be a generically surjective morphism, denote by
$\omega_{\varphi}\in H^0\big(X, \bigwedge^g(\E^*)\otimes\det(\G)\big)$
the associated global section, and by $Z$
its zero scheme.
Suppose that the following conditions hold.
\begin{itemize}
   \item  $Z$ has pure codimension $e-g+1$, and
   \item for every $i\in \{1,\dots, e-g\}$,
$$
H^i\Big(X, \bigwedge^g(\E^*)\otimes\bigwedge^{g+i}\E\otimes S_i(\G^*)\Big)=0.
$$
\end{itemize}
If $\omega\in H^0\big(X, \bigwedge^g(\E^*)\otimes\det(\G)\big)$ is
such that $\omega|_{Z}=0$, then there is an endomorphism
$\alpha\in \End\Big(\bigwedge^g(\E^*)\Big)$
such that $\omega = \alpha \circ \omega_{\varphi}$.
\end{teo}

When $\E$ is fixed, the vanishing assumption of
Theorem~\ref{main_prop} holds provided that $\G$ is sufficiently positive
(see Corollary~ \ref{cor} for a precise statement).
In particular, if
$\F\subset T_X$ is a sufficiently negative
locally free distribution of rank $k$ with singular scheme of pure dimension $k-1$,
and if the vector bundle $\bigwedge^kT_X$ is simple,
then the inclusion $\F\hookrightarrow T_X$ is uniquely determined by its singular scheme.
We remark that a similar approach was used in \cite{C-O-II} to
determine foliations by curves by their singular schemes.

For complex projective spaces,
the problem of determining holomorphic foliations by their singular schemes
has been addressed by several authors.
In \cite{G-MK}, Gomez-Mont and Kempf proved that a foliation by curves $\F$ on $\PP^n$ corresponding to a global section of
$\T_{\PP^n} \otimes \mathcal{O}_{\PP^n}(r)$ is uniquely determined
by its singular scheme $\mathrm{Sing}(\F)$, provided that $r>0$ and $\mathrm{Sing}(\F)$ is reduced.
In \cite{C-O}, Campillo and Olivares showed that when $n=2$ the hypothesis that $\mathrm{Sing}(\F)$ is reduced may be removed.
For $n\geq 3$, they showed in \cite{C-O-II} that $\F$ is uniquely determined
by its singular scheme  provided that $r>0$ and $\mathrm{Sing}(\F)$ is zero-dimensional.

For foliations of rank higher than $1$, much less is known.
Giraldo and Pan-Collantes showed in \cite{GP} that if $\F\subset \T_{\PP^3}$ is a rank $2$ holomorphic
foliation on $\PP^3$ of the form $\F\cong \mathcal{O}_{\PP^n}(a)\oplus \mathcal{O}_{\PP^n}(b)$,
with $a,b\leq -1$, then $\F$ is uniquely determined
by its singular scheme.
Using Theorem~\ref{main_prop} and vanishing results of Manivel,
we  provide sufficient conditions
for distributions of arbitrary rank on $\PP^n$ to be
uniquely determined by their singular schemes,
recovering and generalizing Giraldo and Pan-Collantes' criterion.
(See also Theorems~\ref{main_thm_general}  and
\ref{main_thm}.)

\begin{teo}\label{main_thm_simples}
Let $\F\subset \T_{\PP^n}$ be a holomorphic distribution of rank $k$ on $\PP^n$
with singular scheme of pure dimension $k-1$.
Suppose that one of the following conditions hold:
\begin{enumerate}
    \item $\F$ is locally free and $\F^*\otimes \mathcal{O}_{\PP^n}(k-n)$ is ample, or
    \item $\F$ decomposes as a direct sum of line bundles and $\cF^*\otimes \mathcal{O}_{\PP^n}(k-n+1)$ is ample.
\end{enumerate}
Then $\F$ is uniquely determined by its singular scheme.
\end{teo}

The assumptions of Theorem~\ref{main_thm_simples} are verified in several important cases
(see Remmark~\ref{kupka_simple}).   
The next example  illustrates the necessity of our hypothesis when $k=n-1$.

\begin{exe}
Fix homogeneous polynomials $F_0, \cdots, F_m\in \CC[x_0, \cdots, x_n]$, with degrees $d_i=\deg(F_i)>0$.
Suppose that the hypersurfaces $\{F_i=0\}$  are smooth and in general position.
Set $\Lambda=\big\{(\lambda_0,\cdots, \lambda_m)\in (\CC^*)^{m+1} \ \big| \ \sum_{i=0}^m\lambda_i d_i=0\big\}$.
For each $\lambda \in \Lambda$, we get a foliation $\F_{\lambda}$ of codimension $1$ on $\PP^n$
defined by the $1$-form
$$
\omega_{\lambda}=\sum_{i=0}^m \lambda_iF_0 \cdots \widehat{F_i} \cdots F_m dF_i.
$$
One can check that $\F_{\lambda}\neq \F_{\nu}$ if $\lambda,\nu\in \Lambda$ are not proportional.

By \cite[Theorem 3]{C-S-V}, $\mathrm{Sing}(\F_{\lambda})$ has pure codimension $2$ if and only if the following condition holds:
\begin{equation}\label{conditions_CSV}
m\leq n \ \ \text{and} \ \  d_i=1 \ \ \text{for all} \ \ i.
\end{equation}
If \eqref{conditions_CSV} does not hold, then $\mathrm{Sing}(\F_{\lambda})$ contains isolated points, and thus
$\F_\lambda$ is not locally free.
If \eqref{conditions_CSV} holds, then $\mathrm{Sing}(\F_{\lambda})= \cup_{i\neq j} \{F_i=F_j=0\}$
for all $\lambda\in \Lambda$.
In this case, $\F_{\lambda}\cong \mathcal{O}_{\PP^n}^{\oplus m-1}\oplus \mathcal{O}_{\PP^n}(1)^{\oplus n-m}$,
and $\F_\lambda^*$ is not ample.
\end{exe}

By  \cite[Theorem 2.3]{J}, a generic distribution $\F$ of
codimension 1 on $\PP^n$ with $\deg(\F)<n-1$ has zero
dimensional singular scheme. On the other hand, by \cite[Theorem
2.6]{J}, if $\F$ is a foliation, then its singular scheme has an
irreducible component of codimension 2. Thus, when $n>2$ and
$k=n-1$, Theorem~\ref{main_thm_simples} does not apply to generic
distributions, but rather to foliations. However, the same methods
allow us to show that generic codimension 1 distributions on $\PP^n$
with negative degree are uniquely determined by their singular
schemes.

\begin{teo}\label{thm:generic_codim1_dist}
Let $\F\subset \T_{\PP^n}$ be a holomorphic distribution of rank
$n-1$ on $\PP^n$ with zero dimensional singular scheme.
If $\deg(\F^*)>0$, then $\F$ is uniquely determined by its
singular scheme.
\end{teo}

This paper is organized as follows.
In Section~\ref{section:vbs} we address the problem of determining maps
of vector bundles by their degeneracy schemes.
In Section~\ref{section:foliation} we review the notions of distribution, foliation and
Pfaff field.
In Section~\ref{section:proofs} we apply the results of Section~\ref{section:vbs} to
the problem of recovering holomorphic distributions from their singular schemes.

\

\noindent {\bf {Notation.}}
We identify a vector bundle $\cE$ on a variety $X$
with its locally free sheaf of sections.
When $X=\PP^n$ and $m$ is an integer, we denote by $\cE(m)$ the twisted sheaf
$\cE\otimes \mathcal{O}_{\PP^n}(m)$.

\

\noindent {\bf {Acknowledgements.}} 
The first named author was partially supported by CNPq and Faperj Research Fellowships.
We would like to thank Fernando Cukierman for useful comments on an earlier version of this paper.

%
%

\section{Degeneracy scheme of maps of vector bundles}  \label{section:vbs}
In this section  all varieties are defined over a fixed algebraically closed field $\kk$ of arbitrary characteristic.

Let $X$ be a smooth projective variety, $\E$ and $\G$ locally free sheaves on $X$
of rank $e$ and $g$, respectively, and
$\varphi:\E\to \G$ a generically surjective morphism.
The induced map $\wedge^g\varphi: \bigwedge^g\E \to \det(\G)$
corresponds to a global section  $\omega_{\varphi}\in H^0\big(X, \bigwedge^g(\E^*)\otimes\det(\G)\big)$.


\begin{defi}\label{def:deg}
The \emph{degeneracy scheme} $\mathrm{Sing}(\varphi)$ of the map $\varphi:\E\to \G$ is
the zero scheme of the associated global section
$\omega_{\varphi}\in H^0\big(X, \bigwedge^g(\E^*)\otimes\det(\G)\big)$.
\end{defi}

Next we introduce the key tool to prove Theorem~\ref{main_prop}.

\begin{say}[The Eagon-Northcott resolution]\label{EN}
Let the notation be as above, and suppose that
$Z=\mathrm{Sing}(\varphi)$ has pure expected dimension, i.e., $Z$ has pure codimension equal to
$e-g+1$. Then the structure sheaf of $Z$ admits a special resolution, called the
\emph{Eagon-Northcott resolution} (see for instance \cite[A2.6]{Eisenbud}):

$$
0 \to \bigwedge^e \E\otimes S_{e-g}(\G^*) \otimes \det(\G^*)\to
\bigwedge^{e-1} \E\otimes S_{e-g-1}(\G^*) \otimes \det(\G^*) \to
\ldots
$$
$$
\to \bigwedge^{g+1} \E\otimes \G^* \otimes \det(\G^*) \to
\bigwedge^{g} \E \otimes \det(\G^*) \to \mathcal{O}_{X} \to
\mathcal{O}_{Z}\to 0.
$$
\end{say}

\

\begin{proof}[Proof of Theorem~\ref{main_prop}]
Let $\omega_{\varphi}\in H^0\big(X, \bigwedge^g(\E^*)\otimes\det(\G)\big)$
be the  global section associated to the  generically surjective morphism
$\varphi:\E\to \G$, and $Z$ its zero scheme.
The assumption that  $Z$ has pure codimension $e-g+1$ allows us to consider
the Eagon-Northcott resolution of $\mathcal{O}_Z$ as described in \ref{EN}.
Tensorizing it by $\bigwedge^g(\E^*)\otimes \det(\G)$ we get

$$
0 \to \bigwedge^g(\E^*)\otimes \bigwedge^e \E\otimes S_{e-g}(\G^*) \to
\bigwedge^g(\E^*)\otimes \bigwedge^{e-1} \E\otimes S_{e-g-1}(\G^*) \to
\ldots
$$

$$
\to \bigwedge^g(\E^*)\otimes \bigwedge^{g+1} \E\otimes \G^*  \to
\bigwedge^g(\E^*)\otimes \bigwedge^{g} \E  \to \bigwedge^g(\E^*)\otimes \det(\G)\to
\bigwedge^g(\E^*)\otimes \det(\G)\Big|_{Z}\to 0.
$$

\

Setting $\M_i:=\bigwedge^g(\E^*)\otimes \bigwedge^{i} \E\otimes S_{i-g}(\G^*)$
for $g\leq i\leq e$, the above sequence breaks into short exact sequences:

\begin{equation}\label{seq1}
0 \to\M_e \to \M_{e-1}\to \F_{e-g}\to 0,
\end{equation}
\begin{center}
$\vdots$
\end{center}
\begin{equation}\label{seq2}
0 \to \F_{i-g+2} \to \M_i \to \F_{i-g+1} \to 0,
\end{equation}
\begin{center}
$\vdots$
\end{center}
\begin{equation}\label{seq3}
0\to\F_{2}\to\M_g\to \F_{1}\to 0,
\end{equation}
\begin{equation}\label{seq4}
0\to\F_{1}\to \bigwedge^g(\E^*)\otimes \det(\G) \to \Big(\bigwedge^g(\E^*)\otimes \det(\G)\Big)\big|_{Z}\to 0.
\end{equation}

By assumption $\omega\in H^0\big(X, \bigwedge^g(\E^*)\otimes\det(\G)\big)$ is
such that $\omega|_{Z}=0$. It follows from the sequence
\eqref{seq4} that $\omega$ comes from an element in $H^0(X,\F_1)$.
We will show that $H^1(X,\F_2)=0$.
It will then follow from the sequence \eqref{seq3} that there is an
element
$\alpha\in \End\Big(\bigwedge^g(\E^*)\Big) \cong H^0\Big(\bigwedge^g(\E^*)\otimes \bigwedge^{g} \E\Big)$
such that $\omega = \alpha \circ \omega_{\varphi}$.

By assumption
$$
H^{i}\big(X, \M_{g+i} \big)=0
$$
for $1\leq i\leq e-g$.
Applied to the
cohomology of the sequence \eqref{seq2}, these vanishings
yield inclusions
$$
H^i(X,\F_{i+1})\subset H^{i+1}(X,\F_{i+2})
$$
for $1\leq i\leq e-g-2$.
On the other hand, applied the cohomology of  the sequence \eqref{seq1}, they give
$$
H^{e-g-1}(X,\F_{e-g})= 0.
$$
By descending induction we get $H^1(X,\F_2)=0$, concluding the proof.
\end{proof}

\begin{cor}\label{cor}
Let $X$ be a smooth projective variety, $\E$ and $\G$ locally free sheaves on
$X$ of rank $e$ and $g$, respectively, and $\mathcal{L}$ an ample
line bundle on $X$. There exists an integer $r_0$ such that, for every integer
$r\geq r_0$, the following condition holds.

Let $\varphi:\E\to \G\otimes \mathcal{L}^{\otimes r}$ be a
generically surjective morphism whose degeneracy scheme
$Z=\mathrm{Sing}(\varphi)$ has pure codimension $e-g+1$. If
$$\omega\in H^0\big(X, \bigwedge^g(\E^*)\otimes\det(\G)\otimes
\mathcal{L}^{\otimes rg} \big)$$ is such that $\omega|_{Z}=0$, then
there is an endomorphism $\alpha\in \End\Big(\bigwedge^g(\E^*)\Big)$
such that $\omega = \alpha \circ \omega_{\varphi}$.
\end{cor}

\begin{proof}
By Serre duality,
$
H^i\Big(X, \bigwedge^g(\E^*)\otimes\bigwedge^{g+i}\E\otimes S_i(\G^*)\otimes \mathcal{L}^{*\otimes ri}\Big)
$
is isomorphic to
$H^{\dim(X)-i}\Big(X,\bigwedge^g\E\otimes\bigwedge^{g+i}(\E^*)\otimes S_i\G\otimes \omega_X\otimes \mathcal{L}^{\otimes ri} \Big).
$
It follows from  Serre's vanishing theorem that
there exists an integer $r_0$ such that the latter vanishes
for every integer
$r\geq r_0$ and  every $i\in \{1,\dots, e-g\}$.
So we can apply Theorem~\ref{main_prop}.
\end{proof}

\begin{exe}\label{example:palatini}
Let $m<n$ be positive integers and consider
a \emph{general} injective morphism
$$
\varphi:  \mathcal{O}^{\oplus m}_{\mathbb{P}^n} \to
\Omega^1_{\mathbb{P}^n}(2).
$$
Denote by $Z$ the $(m-1)$-dimensional  degeneracy scheme of $\varphi$, i.e.,
the degeneracy scheme of
the induced map
$T_{\mathbb{P}^n} \to  \mathcal{O}_{\mathbb{P}^n}(2)^{\oplus m}$.
These schemes have been studied by several classical algebraic geometers.
Castelnuovo considered  the case $m = 3$ and $n = 4$  in \cite{castelnuovo}.
In this case $Z$ is the projection in $\PP^4$ of the Veronese surface.
The case $m=3$ and $n=5$ was studied by Palatini in \cite{palatini1}, and then by Fano in \cite{fano}.
In this case $Z$ is a scroll over an elliptic curve.
Palatini considered the case $m=4$ and $n=5$ in \cite{palatini2}.
In this case $Z\subset \PP^5$ is a scroll over a cubic surface in $\PP^3$, known as \emph{Palatini scroll}.
More recently, Hilbert schemes   of these degeneracy schemes were studied by
Faenzi and  Fania in \cite{FF}.

When $m=n-1$, Theorem~\ref{main_prop}  implies that two general injective morphisms
$\phi,\varphi:  \mathcal{O}^{\oplus (n-1)}_{\mathbb{P}^n} \to
\Omega_{\mathbb{P}^n}^1(2)$
have the same degeneracy scheme if and only if
$\phi=\lambda \cdot \varphi$ for some $\lambda\in \kk^*$.
Indeed, by Bott's formulae (see \ref{bott}),
$$
H^1\Big(X,
\Omega^{n-1}_{\mathbb{P}^n}\otimes\bigwedge^{n}T_{\mathbb{P}^n}\otimes
 \mathcal{O}_{\mathbb{P}^n}(-2)^{\oplus (n-1)}\Big)=0,
$$
So Theorem~\ref{main_prop} applies.
Moreover,  $\End(\Omega^{n-1}_{\PP^n})\cong \kk$ (see Lemma~\ref{lema}).
\end{exe}

%
%

\section{Pfaff fields, distributions  and foliations} \label{section:foliation}
In this section $X$ denotes a smooth complex projective variety of dimension $n$.

\begin{defi}
A  \emph{Pfaff field of rank $k$} on $X$ is a nonzero map
$\eta : \Omega^k_X\to \sL$, where  $\sL$ is an
invertible sheaf on $X$.
It corresponds to a global section $\omega_{\eta}\in H^0(X,\bigwedge^k\tx\ox\sL)$.
The \textit{singular scheme} $\mathrm{Sing}(\eta)$ of $\eta$
is the zero scheme of $\omega_{\eta}$.
\end{defi}

\begin{defi}
A \emph{(holomorphic) distribution} of rank $k$ on $X$ is a nonzero coherent subsheaf $\F\subsetneq {T}_X$
of generic rank $k$ which is \emph{saturated}, i.e.,
such that ${T}_X / \F$ is torsion free.

If $\F\subsetneq {T}_X$ is a distribution of rank $k$ on $X$, then $\F$ is reflexive
(see \cite[Remark 2.3]{fano_fol}), and thus $\bigwedge^k \F$ is an invertible sheaf on $X$.
So $\F$ naturaly gives rise to a Pfaff field of rank $k$ on $X$:
$$
\eta_{\F}:\Omega_X^k = \bigwedge^k(T_X^*) \to (\bigwedge^k\F)^*.
$$
The \textit{singular scheme} $\mathrm{Sing}(\F)$ of $\F$ is defined to be
the singular scheme of the associated Pfaff field.

Note that when the distribution $\F$ is locally free, 
$\mathrm{Sing}(\F)$ is precisely the degeneracy scheme of the dual map $\Omega_X^1 \to \F^*$,
as introduced in Definition~\ref{def:deg}.
\end{defi}

\begin{remark}\label{recovering_F}
Let $\F\subsetneq{T}_X$ be a distribution of rank $k$ on $X$, and
$\eta_{\F}:\Omega_X^k \to (\bigwedge^k\F)^*$ the associated Pfaff field.
Then $\F$ may be recovered from $\eta_{\F}$ as follows.
As noted above, $\eta_{\F}$ corresponds to a global section $\omega\in H^0\big(X, \bigwedge^kT_X\ox (\bigwedge^k\F)^*\big)$.
Contraction with $\omega$ yields a morphism
$$
\imath(\omega):\tx \to \bigwedge^{k+1}T_X\ox (\bigwedge^k\F)^*.
$$
Since it comes from a distribution, $\omega$ is locally decomposable.
Therefore, the kernel of $\imath(\omega)$ is a saturated coherent subsheaf of $T_X$
of generic rank $k$.
It is precisely the distribution $\F$.
Here we use the assumption that $\F$ is saturated in $T_X$, and the fact  that two saturated subsheaves of $T_X$
that coincide in a dense open subset  must coincide.  
\end{remark}

\begin{defi}
A distribution on $X$ that is invariant under the Lie Bracket is called an \emph{involutive distribution}.
By abuse of notation, we call such distribution a \emph{(holomorphic) foliation} on $X$. 
\end{defi}

\begin{remark}
Often in the literature, the definition of distribution and foliation  does not require  $\F$ to be saturated in $T_X$.
What we call a foliation is often called a \emph{reduced foliation}.
\end{remark}

\begin{say}[Locally free foliations]\label{kupka_simple}
Let $\F$ be a foliation of rank $k$ on $X$. 
We discuss sufficient conditions  for $\F$ to be locally free.

If $k=1$, then $\F$ is always locally free.
For $k\geq 2$,
we introduce the  \emph{Kupka} condition. 
From the isomorphism $\bigwedge^kT_X\cong \Omega_X^{n-k}\ox \omega_X^*$, we see that 
${\F}$ corresponds to a twisted $(n-k)$-form  $\omega\in H^0\big(X, \Omega_X^{n-k}\ox \omega_X^*\ox (\bigwedge^k\F)^*\big)$.
We say that a point  $P\in \mathrm{Sing}(\F)$ is a \emph{Kupka} singular point if $d  \omega(P)\neq 0$.
Geometrically, this means that, in an analytic neighborhood of $P$, $\F$ is equivalent to the product of a regular (and hence locally free)
foliation and a foliation by curves. 
It follows that $\F$ is locally free at Kupka singular points (see  \cite[Proposition 1.3.1]{medeiros}).

For codimension $1$ foliations, there are other useful conditions.
First suppose that $n=3$ and $k=2$. 
A point  $P\in \mathrm{Sing}(\F)$ is a \emph{generalized Kupka (GK)} singular point if either $d  \omega(P)\neq 0$, or $P$ is an isolated singular point of $d \omega$.
By  \cite[Corollary 2]{calvo_et_al}, $\F$ is locally free at GK singular points. 
Moreover, by \cite[Corollary 1]{calvo_et_al}, the condition of having only GK singularities is open in the space of codimension $1$ foliations on a
smooth projective threefold.
On the other hand, the condition of being locally free   is not always open. 
The notion of GK singularity for codimension $1$ foliations can be generalized to $n\geq 4$ as follows. 
A point  $P\in \mathrm{Sing}(\F)$ is a GK singular point if there exists a threefold $Y\subset X$ smooth at $P$ such that 
either $d  \big(\omega|_Y\big)(P)\neq 0$, or $P$ is an isolated singular point of $d \big(\omega|_Y\big)$ in $Y$.
In this case it is not difficult to see that $\F$ is locally free at $P$.

We also note that the condition of being a direct sum of line bundles is open in the space of codimension $1$ foliations
on $\mathbb{P}^n$ by \cite[Theorem 1]{CP}. 
\end{say}

%
%

\section{Recovering holomorphic distributions from their singular schemes} \label{section:proofs}


Let $X$ be a smooth complex projective variety, and
$\F\subsetneq \T_{X}$ a locally free distribution.
In order to recover $\F$ from $\mathrm{Sing}(\F)$, we
want to apply Theorem~\ref{main_prop} to the dual map $\Omega_X^1 \to \F^*$, and then
Remark~\ref{recovering_F}.
So we need to establish the vanishing required in Theorem~\ref{main_prop}.

In \cite{manivel}, Manivel proved a number of vanishing results for
varieties whose tangent bundle is \emph{uniformly nef}.

\begin{defi}[{\cite[p. 405]{manivel}}]
The category of \emph{uniformly nef} vector bundles on smooth complex projective  varieties is the
smallest category $\sC$ containing tensor products of nef line bundles and Hermitian flat bundles,
closed under direct sums, extensions, quotients, and satisfying the following property.
If $f:Y\to X$ is a finite surjective morphism, and $\E$ is a vector bundle on $X$, then
$\E\in \sC$ if and only if $f^*\E\in \sC$.
\end{defi}

The class of smooth complex projective varieties having uniformly nef tangent bundle includes
projective spaces, abelian varieties, products and finite \'etale covers of those,
among others. For those varieties, we have the following result.

\begin{teo}\label{main_thm_general}
Let $X$ be a smooth complex projective variety with uniformly nef tangent bundle,
and $\F\subset \T_{X}$ a locally free  distribution of rank $k$ on $X$.
Suppose that the following conditions hold.
\begin{itemize}
    \item The singular scheme $\mathrm{Sing}(\F)$ is of pure dimension $k-1$.
    \item There exists an ample line bundle $\sA$ on $X$ such that
        $\F^*\otimes \sA^{-1}$ is nef, and $\omega_X\otimes \sA$ is ample.
\end{itemize}
Let $\eta : \Omega^k_X\to (\bigwedge^k\F)^*$ be a Pfaff field of rank $k$ on $X$,
and suppose that
$\mathrm{Sing}(\F)\subset \mathrm{Sing}(\eta)$.
Then there is an endomorphism $\alpha\in \End(\Omega^k_X)$
such that $\eta = \eta_{\F}\circ \alpha$.
\end{teo}

When $X=\PP^n$, we have the following more refined result.

\begin{teo}\label{main_thm}
Let $\F\subset \T_{\PP^n}$ be a holomorphic distribution of rank $k$ on $\PP^n$
with singular scheme $\mathrm{Sing}(\F)$ of pure dimension $k-1$.
Suppose that one of the following conditions hold:
\begin{enumerate}
    \item $\F$ is locally free and $\F^*\otimes \mathcal{O}_{\PP^n}(k-n)$ is ample, or
    \item $\F$ decomposes as a direct sum of line bundles and $\cF^*\otimes \mathcal{O}_{\PP^n}(k-n+1)$ is ample.
\end{enumerate}
Let $\eta : \Omega^k_{\PP^n}\to (\bigwedge^k\F)^*$ be a Pfaff field of rank $k$ on $\PP^n$, and suppose that
$\mathrm{Sing}(\F)\subset \mathrm{Sing}(\eta)$.
Then $\eta=\lambda \cdot \eta_{\F}$ for some $\lambda\in \CC^*$.
\end{teo}

The following result of Manivel provides the
vanishing required to apply Theorem~\ref{main_prop} in the context of
Theorems~\ref{main_thm_general} and \ref{main_thm}.

\begin{prop}[{\cite[p. 409]{manivel}}]\label{prop:manivel_1}
Let $X$ be a smooth complex projective variety with uniformly nef tangent bundle.
Let $\sL$ be an ample line bundle on $X$, and
$\cE_1$ and $\cE_2$ be nef vector bundles on $X$ of rank $r_1$ and $r_2$, respectively.
Let $i$ and $j$ be positive integers.
Then
$$
H^p\big(X,\Omega_{X}^k\otimes \bigwedge^{i}\cE_1 \otimes S_{j}(\cE_2) \otimes \sL \big) = 0
$$
for every non-negative integer $k$ and every $p > r_1 + r_2 - i - 1$.
\end{prop}

\begin{proof}[Proof of Theorem~\ref{main_thm_general}]
Since $\F$ is locally free, the singular scheme $Z:=\mathrm{Sing}(\F)$
is the degeneracy scheme
of the dual map $\Omega_X^1 \to \F^*$.
Since $Z$ has pure dimension $k-1$, the result will follow from
Proposition~\ref{main_prop} once we prove that
$$
H^i\Big(X, \bigwedge^{k}\T_{X} \otimes \Omega_{X}^{k+i}  \otimes S_{i}(\F)\Big) = 0
$$
for every $i\in\{1,\dots, \dim(X)-k\}$.

By Serre duality,
$$
H^i\Big(X, \bigwedge^{k}\T_{X} \otimes \Omega_{X}^{k+i}  \otimes S_{i}(\F)\Big)
\cong H^{n-i}\Big(X,\Omega^{k}_{X}\otimes \bigwedge^{k+i}\T_{X}
\otimes S_{i}(\F^*)\otimes\omega_X \Big).
$$
Notice that
$$
S_{i}(\F^*)\otimes\omega_X\cong S_i(\F^*\otimes \sA^{-1})\otimes
(\sA^{\otimes i}\otimes \omega_X).
$$
We obtain the required vanishings by applying Proposition~\ref{prop:manivel_1} with
$\E_1= \T_X$, $\E_2=\F^*\otimes \sA^{-1}$, and $\sL=\sA^{\otimes i}\otimes \omega_X$.
\end{proof}

In order to prove Theorem~\ref{main_thm}, we will compute
several cohomology groups of special vector bundles on $\PP^n$.
We start by recalling Bott's formulae.

\begin{say}[Bott's formulae]\label{bott}
Let $n$ be a positive integer, $k$ and $p$ non-negative integers, and $s$ an
arbitrary integer. Then
$$
h^p(\PP^n,\Omega_{\PP^n}^k(s)) =
\begin{cases}
\binom{s+n-k}{s}\binom{s-1}{k} & \text{for } p=0, 0\le k\le n \text{ and } s>k,\\
1 & \text{for } s=0 \text{ and } 0\le k=p\le n,\\
\binom{-s+k}{-s}\binom{-s-1}{n-k} & \text{for } p=n, 0\le k\le n \text{ and } s<k-n,\\
0 & \text{otherwise.}
\end{cases}
$$
\end{say}

In addition to Proposition~\ref{prop:manivel_1}, we will need the following result.

\begin{prop}[{\cite[p.404]{manivel}}]\label{prop:manivel_2}
Let $\cE$ be a vector bundle on $\PP^n$, and $p_0$ an integer.
Suppose that
$$
H^p(\PP^n,\omega_{\PP^n}\otimes \cE(s)) = 0
$$
for every non-negative integer  $s$ and every $p\geq p_0$. Then
$$
H^p(\PP^n,\Omega_{\PP^n}^k\otimes \cE(s)) = 0
$$
for any non-negative integers $s$ and $k$, and
every $p\geq p_0$.
\end{prop}

We are now ready to establish the desired vanishing.

\begin{prop} \label{prop:vanishing}
Let $\cE$ be a vector bundle of rank $k$ on $\PP^n$, and suppose that one of the following
conditions hold.
\begin{enumerate}
    \item $\cE^*(k-n)$ is ample, or
    \item $\cE$ decomposes as a direct sum of line bundles and $\cE^*(k-n+1)$ is ample.
\end{enumerate}
Then
$$
H^p\Big(\PP^n, \bigwedge^{k}\T_{\PP^n} \otimes \Omega_{\PP^n}^i  \otimes S_{i-k}(\cE)\Big) = 0
$$
for any positive integers $i$ and $p$ satisfying $k+1\leq i \leq n$ and $p<i-k+1$.
\end{prop}

\begin{proof}
We may assume that $k\leq n$.
By Serre duality,
$$
H^{p}\Big(\PP^n, \bigwedge^k\T_{\PP^n}\otimes\Omega^{i}_{\PP^n}
\otimes S_{i-k}(\E) \Big)\cong H^{n-p}\Big(\PP^n,\Omega^{k}_{\PP^n}\otimes \bigwedge^i\T_{\PP^n}
\otimes S_{i-k}(\E^*)(-n-1) \Big).
$$

Suppose first that $\cE^*(k-n)$ is ample, and thus $\E^*(k-n-1)$ is nef.
We write
$$
\bigwedge^i\T_{\PP^n}
\otimes S_{i-k}(\E^*)(-n-1)\cong \bigwedge^i\big(\T_{\PP^n}(-1)\big)
\otimes S_{i-k}\big(\E^*(k-n-1)\big)\otimes \mathcal{O}_{\PP^n}(r),
$$
where $r=(n-k+1)(i-k)+i-n-1$.
Since $i\geq k+1$, we get that $r\geq 1$, i.e., $\mathcal{O}_{\PP^n}(r)$ is an ample line bundle.
By applying Proposition~\ref{prop:manivel_1} with
$\E_1= \T_{\PP^n}(-1)$, $\E_2=\E^*(k-n-1)$ and $\sL=\mathcal{O}_{\PP^n}(r)$, we get:
$$
H^{n-p}\Big(\PP^n,\Omega^{k}_{\PP^n}\otimes \bigwedge^i\T_{\PP^n}
\otimes S_{i-k}(\E^*)(-n-1) \Big) = 0
$$
for any positive integers $i$ and $p$ satisfying $k+1\leq i \leq n$ and $p<i-k+1$.

Suppose now that $\cE\cong \bigoplus_{j=1}^{k}\mathcal{O}_{\PP^n}(n-k+a_j)$, with $a_j\geq 0$.
If $i>k+1$, then we write
$$
\bigwedge^i\T_{\PP^n}
\otimes S_{i-k}(\E^*)(-n-1)\cong \bigwedge^i\big(\T_{\PP^n}(-1)\big)
\otimes S_{i-k}\big(\E^*(k-n)\big)\otimes \mathcal{O}_{\PP^n}(r'),
$$
with $r'=(n-k)(i-k)+i-n-1\geq 1$.
The required vanishing then follows from  Proposition~\ref{prop:manivel_1} as above.
From now on we assume that $i=k+1$ and $p=1$.
Since cohomology commutes with direct sum, we must show that
\begin{equation}\label{vanishing}
H^{n-1}\Big(\PP^n,\Omega^{k}_{\PP^n}\otimes \bigwedge^{k+1}\T_{\PP^n}
(a-k-1)\Big) = 0
\end{equation}
for every $a\geq 0$.
By Bott's formulae we have
$$
H^{n-1}\Big(\PP^n,\omega_{\PP^n}\otimes \bigwedge^{k+1}\T_{\PP^n}
(a-k-1+s)\Big)
\cong
H^{n-1}\Big(\PP^n,\Omega^{n-k-1}_{\PP^n}
(a-k-1+s)\Big)
= 0, \ \text{and}
$$
$$
H^{n}\Big(\PP^n,\omega_{\PP^n}\otimes \bigwedge^{k+1}\T_{\PP^n}
(a-k-1+s)\Big)
\cong
H^{n}\Big(\PP^n,\Omega^{n-k-1}_{\PP^n}
(a-k-1+s)\Big)
= 0
$$
for every non-negative integer  $s$.
Proposition~\ref{prop:manivel_2} then implies \eqref{vanishing}.
\end{proof}

It is well known that the tangent bundle of $\PP^n$ is simple, i.e.,
$H^0(\PP^n,\T_{\PP^n}\otimes\Omega^{1}_{\PP^n})\cong \CC$ (see for instance \cite[Lemma 4.1.2]{Okonek}).
The next lemma shows that the same is true for exterior powers of $\T_{\PP^n}$,
providing the last ingredient for the proof of Theorem~\ref{main_thm}

\begin{lema}\label{lema}
Let $n$ be a positive integer, and $k\leq n$ a non-negative integer. Then
$H^0(\PP^n,\bigwedge^k\T_{\PP^n}\otimes\Omega^{k}_{\PP^n})\cong \CC.$
\end{lema}

\begin{proof}
We start with the Euler sequence:
\begin{equation}\label{eq:euler}
0 \to \mathcal{O}_{\PP^n}(-1)\to
\mathcal{O}_{\PP^n}^{\oplus n+1}  \to T_{\PP^n}(-1)\to 0.
\end{equation}
By taking the $(k-i)$-th exterior power and twisting by $\Omega^{k}_{\PP^n}(k-i)$, $0\leq i\leq k-1$, we obtain
the exact sequence
$$
0 \to \bigwedge^{k-i-1}\T_{\PP^n}\otimes\Omega^{k}_{\PP^n}\to
\Omega^{k}_{\PP^n}(k-i)^{\oplus\binom{n+1}{k-i}}  \to\bigwedge^{k-i}\T_{\PP^n}\otimes\Omega^{k}_{\PP^n}\to 0.
$$
By Bott's formula, $H^p\big(\PP^n, \Omega^{k}_{\PP^n}(k-i)\big)=0$ for every $p\geq 0$, while
$H^k\big(\PP^n, \Omega^{k}_{\PP^n}\big)\cong \CC$.
Hence, by taking cohomology of the above sequence we get
\begin{center}
$
H^0(\PP^n,\bigwedge^k\T_{\PP^n}\otimes\Omega^{k}_{\PP^n})\simeq H^1(\PP^n,\bigwedge^{k-1}\T_{\PP^n}\otimes\Omega^{k}_{\PP^n})\simeq
$
$
H^2(\PP^n,\bigwedge^{k-2}\T_{\PP^n}\otimes\Omega^{k}_{\PP^n})\simeq \cdots \simeq H^k(\PP^n,\Omega^{k}_{\PP^n})\simeq \CC.$
\end{center}
\end{proof}

\begin{proof}[Proof of Theorem~\ref{main_thm}]
Let $\eta : \Omega^k_{\PP^n}\to (\bigwedge^k\F)^*$ be a Pfaff field of rank $k$ on $\PP^n$, and suppose that
$\mathrm{Sing}(\F)\subset \mathrm{Sing}(\eta)$.
By Proposition~\ref{prop:vanishing}, we have the required vanishing to apply
Theorem~\ref{main_prop} with $\E=\Omega_{\PP^n}^1$ and $\G=\F^*$.
So we conclude that there is  an endomorphism $\alpha$ of $\Omega^k_{\PP^n}$
such that $\eta = \eta_{\F}\circ \alpha$.
On the other hand, by Lemma~\ref{lema},
$\End(\Omega^{k}_{\PP^n})\cong H^0(\PP^n,\bigwedge^k\T_{\PP^n}\otimes\Omega^{k}_{\PP^n})\cong \CC.$
Hence $\eta=\lambda \cdot \eta_{\F}$ for some $\lambda\in \CC^*$.
\end{proof}

\begin{proof}[Proof of Theorem~\ref{thm:generic_codim1_dist}]
Let $r\in \ZZ$ be such that $\mathcal{O}_{\PP^n}(r)= \det(\F^*)\otimes \omega_{\PP^n}^*$.
The condition $\deg(\F^*)>0$ implies that $r>n+1$.
The distribution $\F$ induces a morphism
$\T_{\PP^n}\rightarrow \mathcal{O}_{\PP^n}(r)$, which
corresponds to a nonzero section $\omega_{\F}\in H^0\big(X,\Omega^1_{\PP^n}(r))$.

By twisting the dual of the Euler sequence \eqref{eq:euler} by
$\bigwedge^{i+1}\T_{\PP^n}(-ir+1)$, $1\leq i\leq n-1$, we obtain the
exact sequence
$$
0 \to \Omega^{1}_{\PP^n} \otimes \bigwedge^{i+1}\T_{\PP^n}(-ir)\to
\bigwedge^{i+1}\T_{\PP^n}(-ir-1)^{\oplus(n+1)} \to
\bigwedge^{i+1}\T_{\PP^n}(-ir)\to 0.
$$
By taking cohomology of the above sequence we get
$$
\cdots  \rightarrow H^{i-1}(\PP^n,
\bigwedge^{i+1}\T_{\PP^n}(-ir))\rightarrow
H^{i}(\PP^n,\Omega^{1}_{\PP^n} \otimes
\bigwedge^{i+1}\T_{\PP^n}(-ir)) \rightarrow
$$
$$
 \rightarrow
H^{i}(\PP^n,\bigwedge^{i+1}\T_{\PP^n}(-ir-1)^{\oplus(n+1)})\rightarrow \cdots
$$
By Bott's formulae, for $1\leq i\leq n-1$,
$$
H^{i-1}(\PP^n, \bigwedge^{i+1}\T_{\PP^n}(-ir))=
H^{i}(\PP^n,\bigwedge^{i+1}\T_{\PP^n}(-ir-1)^{\oplus(n+1)})=0.
$$
Thus $H^{i}(\PP^n,\Omega^{1}_{\PP^n} \otimes
\bigwedge^{i+1}\T_{\PP^n}(-ir))=0$ for $1\leq i\leq n-1$.

Let $\omega\in H^0\big(X,\Omega^1_{\PP^n}(r))$ be such that $\omega|_{Sing(\F)}=0$.
It follows from Theorem \ref{main_prop} and
Lemma \ref{lema} that there is
$\alpha\in \mathbb{C}$ such that $\omega = \alpha \cdot \omega_{\F}$.
\end{proof}

\end{document}